\documentclass[11pt,letterpaper]{amsart}
\usepackage{mathtools, bbm} 
\usepackage{amsmath}
\usepackage{mathrsfs}
\usepackage{amssymb}
\usepackage{xcolor}
\usepackage{amscd} 
\usepackage{tikz}
\usepackage{graphicx}
\usepackage{subcaption}

\usepackage{hyperref}
\usepackage{cleveref} 
\usepackage{comment}
\usepackage[margin=1in]{geometry} 

\usepackage{amsthm}
 \newtheorem{thm}{Theorem}[section]
 \newtheorem{prop}[thm]{Proposition}
 
\theoremstyle{definition}
 \newtheorem{ex}[thm]{Example}
 \newtheorem{dfn}[thm]{Definition}
\theoremstyle{remark}
 \newtheorem{rem}[thm]{Remark}

 \numberwithin{equation}{section}
\DeclareMathOperator{\spn}{span}
\DeclareMathOperator{\rk}{rank}

\DeclareMathOperator{\inv}{inv}

\DeclareMathOperator{\Des}{Des}
\DeclareMathOperator{\Asc}{Asc}

\DeclareMathOperator{\sym}{\mathfrak{S}}

\newcommand\coveredby{\mathrel{\ooalign{$<$\cr
  \hidewidth\raise0.0ex\hbox{$\cdot\mkern2mu$}\cr}}}
\newcommand\covers{\mathrel{\ooalign{$>$\cr
  \hidewidth\raise0.0ex\hbox{$\cdot\mkern7mu$}\cr}}}

\renewcommand{\le}{\leqslant}\renewcommand{\leq}{\leqslant}
\renewcommand{\ge}{\geqslant}
\renewcommand{\setminus}{\smallsetminus}

 \author[Li and Sundaram]{Yifei Li and Sheila Sundaram}

\address{Yifei Li: University of Illinois at Springfield, Springfield, IL 62703  USA} 
\email{yli236@uis.edu}
 \address{Sheila Sundaram:  University of Minnesota, Minneapolis, MN 55455, USA}
\email{shsund@umn.edu}

\title{Segre powers of posets preserve EL-shellability}
\date{\today}

\begin{document}


\begin{abstract} For a bounded and graded poset $P$, we show that if  $P$ is EL-shellable, then so is its $t$-fold Segre power $P^{(t)}=P\circ \cdots \circ P$ ($t$ factors), as  defined by Bj\"orner and Welker [J. Pure Appl. Algebra, 198(1-3), 43--55 (2005)]. Our EL-labeling  leads to formulas for the rank-selected invariants of  $P^{(t)}$,    generalising those given   by Stanley for the subspace lattice [J. Combinatorial Theory Ser. A, 20(3):336-356, 1976]. 
\end{abstract}

\maketitle

\section{Introduction}\label{sec:Intro}
The \emph{Segre product}  of  posets   was first defined by Bj\"orner and Welker, who showed \cite[Theorem 1]{Segre_rees} that this operation preserves the property of being homotopy Cohen-Macaulay.  The  Segre square $P\circ P$ was studied by the first author in \cite{YLiqCSV2023}, when $P$ is the Boolean lattice $B_n$ and the subspace lattice $B_n(q)$.

The work of \cite{YLiqCSV2023} was further developed in \cite{LiSu-ECA2025}. Inspired by these papers, the present work studies the EL-shellability of the $t$-fold Segre power $P^{(t)}=P\circ \cdots\circ P$ ($t$ factors) of a pure poset $P$, for arbitrary $t$. 

The $t$-fold Segre power of the subspace lattice $B_n(q)$ appears in an early paper of Stanley \cite[Ex. 1.2]{RPS-BinomialPosetsJCTA1976}, where the concept of a binomial poset is introduced and studied.   As pointed out in \cite[Example 10]{Segre_rees}, the Segre power is a particular example of a binomial poset.  See also \cite[Ex.3.18.3]{ec1}.

The present paper is organized as follows.  After  a brief review of poset topology, in Section 2 we present our main results. 
We show in ~\Cref{thm:Segre-t_EL} that if $P$ admits an EL-labeling, then so does $P^{(t)}$. We describe the increasing and decreasing chains of $P^{(t)}$ in terms of those of $P$, giving general combinatorial formulas for the rank-selected Betti numbers of $P^{(t)}$ in ~\Cref{thm:mu-P-to-t-Segre-of-P-rank-sel}.  Finally in Section~\ref{sec:tfold-subsp-lattice} we discuss the special case of the lattice $B_n(q)$ of subspaces of an $n$-dimensional vector space over a field with $q$ elements.  We apply ~\Cref{thm:Segre-t_EL}
to obtain an explicit EL-labeling of  the $t$-fold Segre power of 
$B_n^{(t)}(q)$.

\section{Segre powers preserve EL-shellability}\label{sec:SegreEL}

For a positive integer $n$, let $[n]$ denote the set $\{1,\ldots,n\}$.  Good general references for poset topology are \cite{BjTopMeth1995} and \cite{WachsPosetTop2007}. 

A poset $P$ is 
\begin{enumerate}
\item \emph{bounded}
 if it has a unique minimal, denoted by  $\hat 0$, and a unique maximal element, denoted by $\hat 1$;  
 \item \emph{graded} if in any interval $(x,y)$, every maximal chain (with respect to inclusion) has the same length.
\end{enumerate}

If the poset $P$ is bounded and graded, it has a \emph{rank function}  $\rk(x)$, defined as the length of any maximal  chain from  $\hat 0$ to $x$.  The \emph{rank of $P$} is then the rank of $\hat 1$.

Recall \cite{ec1} that the product poset $P\times Q$ of two posets $P,Q$ has order relation defined by $(p,q)\le (p',q') $ if and only if $p\le_P p'$ and $q\le_Q q'$. Segre products are defined in greater generality by Bj\"orner and Welker  in \cite{Segre_rees}. As in \cite{YLiqCSV2023} and \cite{LiSu-ECA2025}, this paper is concerned with the following special case. 
\begin{dfn}[\cite{Segre_rees}]\label{def:Segre-prod} 
 Let $P$ be a bounded and graded poset. The $t$-fold Segre power ${P\circ \cdots \circ P}$ ($t$ factors), denoted $P^{(t)}$,  is  defined  for all $t\ge 2$ to be the induced subposet of the $t$-fold product poset $P\times \cdots \times P$ ($t$ factors) consisting of $t$-tuples $(x_1,\ldots, x_t)$ such that $\rk(x_i)=\rk(x_j), 1\le i,j\le t$.  The cover relation in $P^{(t)}$ is thus 
$(x_1,\ldots,x_t) \coveredby (y_1,\ldots, y_t)$ if and only if $x_i\coveredby y_i$ in $P$, for all $i=1,\ldots, t$.   When $t=1$ we set $P^{(1)}$ equal to $P$.
\end{dfn}
It follows that $P^{(t)}$ is also a bounded and graded poset which inherits the rank function of $P$. 

 Figure~\ref{fig:Segre square}, reproduced from \cite{LiSu-ECA2025} for convenience, shows the Segre square $P\circ P$ of a poset $P$,  an induced subposet of  $P\times P$. 
Missing in  $P\circ P$ are   these elements in the product $P\times P$:
$(a,c), (a,d), (b,c), (b,d), (c,a), (c,b),$ $ (d,a), (d,b)$, 
as well as all $(\hat 0, y), (y, \hat 0), y\ne \hat 0$, and all $(x, \hat 1), (\hat 1, x), x\ne \hat 1$.
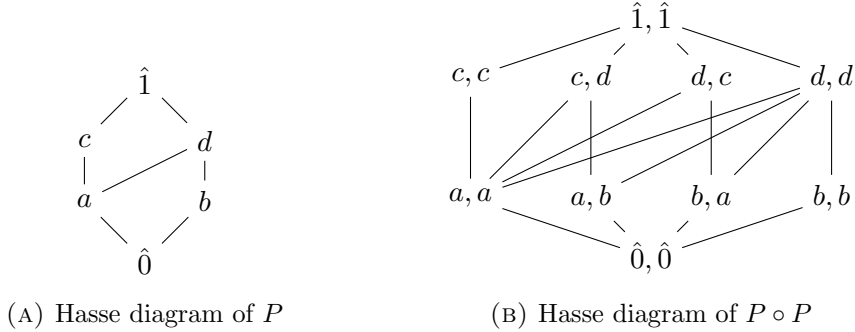
\begin{figure}\
    \begin{subfigure}{0.3\textwidth}
        \centering
        \begin{tikzpicture}[scale=0.8]
        \node (max) at (0,0) {$\hat{1}$};
        \node (c) at (-1,-1) {$c$};
        \node (d) at (1,-1) {$d$};
        \node (a) at (-1,-2) {$a$};
        \node (b) at (1,-2) {$b$};
        \node (min) at (0,-3) {$\hat{0}$};
        \draw (min)--(a)--(c)--(max)--(d)--(a);
        \draw (min)--(b)--(d);
    \end{tikzpicture}
    \caption{Hasse diagram of $P$}
    \end{subfigure}
    \begin{subfigure}{0.5\textwidth}
        \centering
        \begin{tikzpicture}[scale=0.8]
        \node (max) at (0,0) {$\hat{1}, \hat{1}$};
        \node (cc) at (-3,-1) {$c,c$};
        \node (cd) at (-1,-1) {$c,d$};
        \node (dc) at (1,-1) {$d,c$};
        \node (dd) at (3,-1) {$d,d$};
        \node (aa) at (-3,-3) {$a,a$};
        \node (ab) at (-1,-3) {$a,b$};
        \node (ba) at (1,-3) {$b,a$};
        \node (bb) at (3,-3) {$b,b$};
        \node (min) at (0,-4) {$\hat{0}, \hat{0}$};
        \draw (min)--(aa)--(cc)--(max)--(cd)--(aa)--(dc)--(max);
        \draw (aa)--(dd);
        \draw (ab)--(cd);
        \draw (ba)--(dc);
        \draw  (min)--(ab)--(dd)--(max);
        \draw  (min)--(ba)--(dd);
        \draw (min)--(bb)--(dd);
    \end{tikzpicture}
    \caption{Hasse diagram of $P\circ P$}
    \end{subfigure}
\caption{$P\circ P$ is an induced subposet of the product poset $P\times P$}
\label{fig:Segre square}
\end{figure}

An \emph{edge labeling} of a bounded poset $P$ is a map $\lambda: \mathcal E(P)\longrightarrow \Lambda$, where $\mathcal E(P)$ is the set of covering relations $x\coveredby y$ of $P$ and $\Lambda$ is some poset. A maximal chain $c=(x\coveredby x_1\coveredby \cdots \coveredby x_r\coveredby y)$ in the interval $[x,y]$, where $x_0=x$ and $x_{r+1}=y$, 
has an \emph{ascent} (respectively, \emph{descent}) at some $i\in [r]$ 
if $\lambda(x_{i-1},x_i)<\lambda(x_{i},x_{i+1})$ (respectively, $\lambda(x_{i-1},x_i)\not<\lambda(x_{i},x_{i+1})$). 
  The chain $c$ in $[x,y]$ 
is \emph{increasing} if $\lambda (x,x_1)<\lambda(x_1,x_2)<\cdots<\lambda(x_r,y)$. A chain $c$ is called \emph{decreasing} if there is no $i\in [r]=\{1,2,\dots,r\}$ such that $\lambda(x_{i-1},x_i)<\lambda(x_i,x_{i+1})$ in $\Lambda$. 

The descent set $\Des(c)$ of the chain $c$ is then the set $\{i\in [r-1]:  \lambda(x_{i-1},x_i)\not<\lambda(x_i,x_{i+1})\}$.
Each maximal chain $c$ in $[x,y]$ is associated, via its edge labels,  with a word $$\lambda (c)=\lambda (x,x_1)\lambda(x_1,x_2)\cdots\lambda(x_r,y).$$
For two maximal chains $c_1, c_2$ in $[x,y]$, if the word $\lambda (c_1)$ lexicographically precedes the word $\lambda (c_2)$, we say that the chain $c_1$ lexicographically precedes the chain $c_2$ and we denote this by $c_1<_Lc_2$.  

Define the descent set of a word $w_1\cdots w_r$ to be the set $\Des(w_1\cdots w_r):=\{i\in [r-1]:  w_i\not<w_{i+1}\}$, so that $\Des(c)=\Des(\lambda(c))$.  Also define $\Asc(c)=\Asc(\lambda(c))$ to be the set of  ascents in the (label sequence of) the chain $c$.  Thus $\Des(c)=[r-1]\setminus \Asc(c)$, where $r$ is the length of the chain $c$.

Finally, by a maximal chain in $P$ we will mean a maximal chain from the unique least element $\hat 0$ of $P$ to the unique greatest element $\hat 1$ of $P$.

\begin{dfn}[{\cite[Definition 2.1]{BW2}, \cite[Definition 3.2.1]{WachsPosetTop2007}}]
\label{def:EL} An edge labeling is called an \emph{EL-labeling} (edge lexicographical labeling) if for every interval $[x, y]$ in $P$,
\begin{enumerate}
\item there is a unique increasing maximal chain $c$ in $[x, y]$, and
\item $c<_Lc'$ for all other maximal chains $c'$ in $[x, y]$.
\end{enumerate}
If the poset $P$ admits an EL-labeling, it is called \emph{EL-shellable}.
\end{dfn}

\begin{dfn}[{\cite{Bac1980, BjTopMeth1995}}] Let $P$ be a  bounded and graded  poset, with rank function $\mathrm{rk}$,  and let $\mathbbm{k}$ be any field.  Then $P$ is \emph{Cohen-Macaulay} over $\mathbbm{k}$ if for every open interval $(x,y)$, the reduced simplical homology $\tilde{H}_i(x,y)$ of $(x,y)$  {vanishes in all but the top dimension} $\mathrm{rk}(y)-\mathrm{rk}(x)-2$.  In particular this is true for $P=(\hat 0, \hat 1)$ itself.
\end{dfn}

\begin{thm}[{\cite{BjTopMeth1995, BW2, BjWachsTAMSII1996}, \cite[Proposition 3.8.6]{ec1}\label{thm:poset-topology-mu}}]  Let $P$ be a bounded and graded  poset, of rank $n$. Let $\mu(P)$ denote the M\"obius number of $P$. 
The following quantities are equal:
\begin{enumerate}
    \item $(-1)^{n-2}\mu(P)$.
    \item  $(-1)^{n-2}$ times the reduced Euler characteristic of the order complex of $P$.
\end{enumerate}
     If in addition $P$ has an EL-labeling, then the order complex of the poset $P$ has the homotopy type of a wedge of spheres in the top dimension $n-2$, and the poset is Cohen-Macaulay.  In particular, all homology vanishes except possibly in the top dimension, and the above numbers are equal to each of the following.
\begin{enumerate}
    \item[(3)]The dimension of the top homology of the order complex of $P$.
    \item[(4)] The number of spheres in  the order complex of $P$.
    \item [(5)]The number of decreasing maximal chains in the EL-labeling.
\end{enumerate}
\end{thm}

Let $P$ be a finite graded bounded poset of rank $n$, and let $J\subset [n-1]=\{1,\ldots,n-1\}$ be any subset of nontrivial ranks.  Let $P_J$ denote  the rank-selected bounded subposet $P_J$ of $P$ consisting of elements in the rank-set $J$, together with $\hat 0$ and $\hat 1$.
Stanley  \cite[Section 3.13]{ec1} defined two rank-selected invariants $\tilde\alpha_P(J)$ and $\tilde\beta_P(J)$ 
as follows.   Then 

\begin{itemize}
    \item
$\tilde\alpha_P(J)$ is the number of maximal chains in the rank-selected subposet $P_J$, and 
\item
$\tilde\beta_P(J)$ is the integer defined by the equation 
\[\tilde\beta_P(J):=\sum_{U\subseteq J} (-1)^{|J|-|U|}\tilde\alpha_P(U).\]  
Equivalently, 
\[\tilde\alpha_P(J)=\sum_{U\subseteq J} \tilde\beta_P(U).\]
\end{itemize}
One also has the formula for the M\"obius number of the rank-selected subposet $P_J$ \cite[Eqn. (3.54)]{ec1}:
\begin{equation}\label{eqn:rank-sel-inv-to-mu}\tilde\beta_P(J)=(-1)^{|J|-1} \mu_{P_J}(\hat 0, \hat 1).\end{equation}

\begin{rem}\label{rem:rank-invariants-EL}
    If the poset $P$ has an edge-labeling $\lambda$ which is an EL-labeling, then because of condition (2) in Definition~\ref{def:EL}, one sees that 
 $\tilde\alpha_P(J)$ equals the number of maximal chains in $P$ whose descent set with respect to the labeling $\lambda$ is contained in $J$, and $\tilde\beta_P(J)$ is the number of maximal chains in $P$ with descent set exactly $J$.  For details see \cite[Proof of Theorem 5.2]{Wachs-d-div-AIM1996}, \cite[Theorem 2.2]{BjGarsiaStanley-CM-1982}, and also  the R-labeling version in \cite[Theorem 3.14.2]{ec2}.
\end{rem}
  It is known that rank-selected subposets of a Cohen-Macaulay poset are also Cohen-Macaulay \cite[Theorem~6.4]{Bac1980}.
 
Again let $P$ be a graded and bounded poset of rank $n$, with an edge labeling $\lambda: \mathcal E(P)\longrightarrow \Lambda$ for some poset $\Lambda$. 
Let $\Lambda^t$ denote the $t$-fold product poset ${\Lambda\times\cdots\times\Lambda}$ ($t$ factors). 
 The edge labeling $\lambda$ extends  to an edge labeling 
 $\lambda^{(t)}:\mathcal E(P^{(t)})\longrightarrow \Lambda^t$ 
 of the $t$-fold Segre power $P^{(t)}$, by setting 
 \[\lambda^{(t)}((x_1,\ldots,x_t)\coveredby (y_1,\ldots,y_t))
 =(\lambda(x_1\coveredby y_1), \ldots, \lambda(x_t\coveredby y_t)).\]

 The label of a maximal chain $C$ in $P^{(t)}$ from $X=(x_1, \ldots, x_t)$ to $Y=(y_1, \ldots, y_t)$, 
 \[X=X_0\coveredby X_1\coveredby \cdots \coveredby X_r=Y, \] is  the concatenated string of $t$-tuples of edge labels $\lambda^{(t)}(X_i, X_{i+1}).$

 A maximal chain in the Segre power $P^{(t)}$ from $X=(x_1, \ldots, x_t)$ to $Y=(y_1, \ldots, y_t)$ is in bijection with a $t$-tuple of chains $(c_1, \ldots, c_t)$ where each $c_i$ is a maximal chain from $x_i$ to $y_i$. 
 
To illustrate these ideas, consider the 
 Boolean lattice $B_n$ of subsets of an $n$-element set, which  has a well-known EL-labeling defined by $\lambda(A\coveredby B)= a$ where $B\setminus A=\{a\}$; see  \cite[Section 3.2.1]{WachsPosetTop2007}.  In this case the poset $\Lambda$ is the totally ordered set $[n]$.
 
 \begin{ex}\label{ex:t-Segre-chain-label} Let $P=B_4$ be the Boolean lattice of subsets of the set $[4]=\{1,2,3,4\}$, and consider the following maximal chain in $P^{(3)}=P\circ P\circ P$, with labels in the product poset $[4]\times [4]\times[4]$ shown above each cover relation:

 \[ C=(\hat 0 \overset{(1,3,3)}{\coveredby} (1,3,3) \overset{(2,2,1)}{\coveredby} (12, 23, 13) \overset{(3,1,4)}{\coveredby} (123, 123, 134) \overset{(4,4,2)}{\coveredby} \hat 1),\]
 omitting the braces for the subsets for ease of reading; thus $(12, 23, 13)$ represents the 3-tuple of subsets $(\{1,2\}, \{2,3\}, \{1,3\})$.

 Then according to the definition of edge labeling, the label of the chain $C$ is the following concatenation of 3-tuples, or equivalently a ``word" in the alphabet of 3-tuples:
 \[\lambda^{(3)}(C)=(1,3,3)(2,2,1)(3,1,4)(4,4,2).\]

    The chain $C$ maps bijectively to the 3-tuple of chains $(c_1, c_2, c_3)$ where 
    \[c_1=(\hat 0\overset{1}{\coveredby} 1\overset{2}{\coveredby} 12 \overset{3}{\coveredby} 123 \overset{4}{\coveredby} \hat 1),
     c_2=(\hat 0\overset{3}{\coveredby} 3\overset{2}{\coveredby} 23 \overset{1}{\coveredby} 123 \overset{4}{\coveredby} \hat 1),
     c_3=(\hat 0\overset{3}{\coveredby} 3\overset{1}{\coveredby} 13 \overset{4}{\coveredby} 134 \overset{2}{\coveredby} \hat 1).\]
     The EL-labels of these chains $c_i$ in $P=B_4$ are the permutations $1234, 3214, 3142$.

     Now consider another chain, $D$, in $P^{(3)}$, 
     \[ D=(\hat 0 \overset{(1,3,3)}{\coveredby} (1,3,3) \overset{(2,4,4)}{\coveredby} (12, 34, 34) \overset{(3,1,1)}{\coveredby} (123, 134, 134) \overset{(4,2,2)}{\coveredby} \hat 1).\]
     We have 
     \[\lambda^{(3)}(D)=(1,3,3)(2,4,4)(3,1,1)(4,2,2).\]
      The chain $D$ maps bijectively to the 3-tuple of chains $(d_1, d_2, d_3)$ where 
    \[d_1=(\hat 0\overset{1}{\coveredby} 1\overset{2}{\coveredby} 12 \overset{3}{\coveredby} 123 \overset{4}{\coveredby} \hat 1),
     d_2=(\hat 0\overset{3}{\coveredby} 3\overset{4}{\coveredby} 34 \overset{1}{\coveredby} 134 \overset{2}{\coveredby} \hat 1),
     d_3=(\hat 0\overset{3}{\coveredby} 3\overset{4}{\coveredby} 34 \overset{1}{\coveredby} 134 \overset{2}{\coveredby} \hat 1).\]
     The EL-labels of these chains $d_i$ in $P=B_4$ are the permutations $1234, 3412, 3412$.
     
     We see that $\lambda^{(3)}(C)$ precedes $\lambda^{(3)}(D)$ in the lexicographic order on the words in $[4]^3=[4]\times[4]\times[4]$,
     because $(2,2,1)<(2,4,4)$ in the product poset $[4]^3$.

     The chain $C$ has no ascents and is therefore a decreasing chain.  The chain $D$ has an ascent in position 1 and 3, since $(1,3,3)<(2,4,4)$ and $(3,1,1)<(4,2,2)$ in the poset  $[4]^3$.  Notice that the chains $c_1, c_2, c_3$ have no common ascents, while  1 and 3 are the only common ascents of the chains $d_1, d_2, d_3$. 
 \end{ex}
 
This extension of the labeling from $P$ to $P^{(t)}$ has the following important property, our first  result.
 \begin{thm} \label{thm:Segre-t_EL} Let $\lambda: \mathcal E(P)\longrightarrow \Lambda$ be an EL-labeling of the graded and bounded poset $P$ of rank $n$. Then $\lambda^{(t)}:\mathcal{E}(P^{(t)})\longrightarrow \Lambda^{t}$ is an EL-labeling of the $t$-fold Segre power $P^{(t)}$.
Identifying a maximal chain $C$ in $P^{(t)}$ with the $t$-tuple of maximal chains $(c_1, \ldots, c_t)$, where $c_i$ is a maximal chain in $P$,   we have the following.
\begin{enumerate}
\item  $j\in [n-1]$ is a descent of the $t$-tuple $(c_1,\ldots,c_t)$ if and only if $j$ is \emph{not} a common ascent of $ c_1, \ldots, c_t$.  Hence   the descent set $\Des(c_1,\ldots,c_t)$ is the complement in $[n-1]$ of  $\cap_{i=1}^t \Asc(c_i)$. 
\item
The decreasing maximal chains of $P^{(t)}$ with respect to the EL-labeling $\lambda^{(t)}$ are precisely the $t$-tuples of maximal chains $(c_1,\ldots,c_t)$ in $P$ with no common ascent.

\item 
The rank-selected invariant $\tilde\beta_{P^{(t)}}(J)$ 
 is the number of $t$-tuples   $(c_1,\ldots,c_t)$ of maximal chains $c_i$ in $P$ 
such that $[n-1]\setminus J$ is precisely the set of common ascents of $ c_1, \ldots, c_t$. 
\item   The rank-selected invariant 
    $\tilde\alpha_{P^{(t)}}(J)$ is the number of $t$-tuples   $(c_1,\ldots,c_t)$ of maximal chains $c_i$ in $P$ such that $J$ contains the descent set of $(c_1, \ldots, c_t)$. 
\end{enumerate}
\end{thm}
\begin{proof}
Let $[X,Y], X\ne Y,$ be a closed interval in $P^{(t)}$, where $X=(x_1,\dots,x_t)$, $Y=(y_1,\dots,y_t)$ and $x_i< y_i$ in $P$ for all $i\in [t]$. Two important observations about 
 chains in $[X,Y]$ and the labeling $\lambda^{(t)}$ follow, the first of which has already been made.
 \begin{enumerate}
 \item[(i)]
Maximal chains $C$ in the interval $[X,Y]$  are  in bijection with $t$-tuples of maximal chains $(c_1,\ldots,c_t)$ where $c_j$ is a maximal chain in the interval $[x_j,y_j]$ of $P$. 
\item[(ii)] Recall that the EL-label $\lambda^{(t)}(C)$ is a word in the alphabet $\Lambda^t$ of $t$-tuples of elements of $\Lambda$. By definition of the labeling $\lambda^{(t)}$, the chain $C$ corresponding to the $t$-tuple of chains $(c_1,\ldots,c_t)$ in $P$ has an ascent  at $i$  if and only if every chain $c_j$ in $[x_j,y_j]$, $1\le j\le t$,   has an ascent  at $i$. This is because of the order relation in the product poset $\Lambda^t$: $(a_1,\ldots, a_t)< (b_1,\ldots, b_t)$ if and only if $a_i\le b_i$ for all $i$, and there is at least one $i$ such that $a_i\ne b_i$. Furthermore, the chain $C$ corresponding to the $t$-tuple of chains $(c_1,\ldots,c_t)$ in $P$ has a descent  at $i$  if and only if there is at least one chain $c_j$ in $[x_j,y_j]$, $1\le j\le t$,   having a descent  at $i$. 
See Example~\ref{ex:t-Segre-chain-label}. 
\end{enumerate}

Since  $\lambda$  is an EL-labeling for $P$, there is a unique increasing maximal chain $c_i$ in each $[x_i,y_i]$ that lexicographically precedes all other chains in the same interval. Then the chain in $[X,Y]$ 
which is in bijection with the $t$-tuple  $(c_1,\ldots,c_t)$ 
must be the unique increasing maximal chain in $[X,Y]$. Any other chain has non-increasing labels (see Definition~\ref{def:EL}) in some $[x_i,y_i]$, hence is non-increasing in $[X,Y]$. This unique increasing maximal chain of $[X,Y]$ must also satisfy Part $(2)$ of Definition \ref{def:EL} because $c_i$ satisfies this condition for all $i\in [t]$. Hence $\lambda^{(t)}$ is an EL-labeling for $P^{(t)}$. 

Again from Definition \ref{def:EL}, a decreasing chain $C$ in $P^{(t)}$ has no ascents. By Item (ii), in the corresponding $t$-tuple of chains $(c_1,\ldots,c_t)$,  the chains $c_i$  have no common ascents. Parts (1) and (2) now follow.

By Remark~\ref{rem:rank-invariants-EL}, the rank-selected invariant $\tilde\beta_{P^{(t)}}(J)$ is the number of  maximal chains $(c_1,\ldots,c_t)$ in  $P^{(t)}$ with descent set exactly equal to $J$. By Item (ii) above again, this means no index $j\in J$ is a common ascent of 
 all $t$ chains $c_1,\ldots, c_t$, and conversely.  This establishes  Part (3), and hence also Part (4).
\end{proof}
The next observation is motivated by the features of two particular posets, the Boolean lattice $B_n$ of subsets of $[n]$ with respect to inclusion and the subspace lattice $B_n(q)$ of subspaces of the $n$-dimensional vector space $\mathbb{F}^n_q$ over a finite field of $q$ elements.  Both are modular geometric lattices; see \cite[Ex. 3.10.2]{ec1}.  The subspace lattice is a $q$-analogue of the Boolean lattice in the sense of Simion \cite{Simion1995}.

Let $P$ be a graded and bounded poset of rank $n$ with an EL-labeling  
$\lambda: \mathcal{E}(P)\rightarrow \Lambda$ for some poset $\Lambda$.  Let $w_n(\Lambda)$ denote the set of words of length $n$ in the alphabet $\Lambda$. 
 Then for every maximal chain $c=(\hat{0}\coveredby x_1\coveredby \cdots \coveredby x_{n-1}\coveredby\hat{1})$ in $P$, the word $\lambda(c)$  is an element of $w_n(\Lambda)$.  We write $w_n(\Lambda)^{\times t}$ for the $t$-fold Cartesian product $\underbrace{w_n(\Lambda)\times\cdots\times w_n(\Lambda)}_t.$

Define an integer-valued function $e\ell_\lambda$ on $w_n(\Lambda)$ by \[e\ell_\lambda(\sigma):=|\{\text{maximal chains $c$ in $P$}: \lambda(c)=\sigma\}|.\]

By Theorem~\ref{thm:poset-topology-mu}, for a poset with an EL-labeling, all homology vanishes except possibly in the top dimension.
From Theorem~\ref{thm:Segre-t_EL} we deduce the following formulas for the rank-selected M\"obius numbers, or equivalently the rank-selected Betti numbers of the $t$-fold Segre power of a poset $P$ with EL-labeling $\lambda$.  
\begin{thm}\label{thm:mu-P-to-t-Segre-of-P-rank-sel} With $P$ of rank $n$ as above, for each rank subset $J\subseteq [n-1]$, we have the following formulas for the dimension of the top homology module in each case:
\begin{enumerate}
    \item  
{$(-1)^{|J|-1}\mu(P_J)=\sum_{\substack{\sigma\in w_n(\Lambda)\\ \Des(\sigma)=J}}
e\ell_\lambda(\sigma),$} and hence

    $(-1)^{n-2}\mu(P)=\sum_{\substack{\sigma\in w_n(\Lambda)\\ 
    \Des(\sigma)=[n-1]}} e\ell_\lambda(\sigma)$.
    \item $(-1)^{|J|-1}\mu(P_J^{(t)})=\sum_{\substack{(\sigma^1,\ldots, \sigma^t)\in w_n(\Lambda)^{\times t}\\ \Des((\sigma^1,\ldots, \sigma^t))=J}} \prod_{i=1}^t e\ell_\lambda(\sigma^i)$,
and hence 

 $     (-1)^{n-2}\mu(P^{(t)})
    =\sum_{\substack{(\sigma^1,\ldots, \sigma^t)\in w_n(\Lambda)^{\times t}\\
    {\sigma^1,\ldots, \sigma^t \text{have no common ascents}} }}\prod_{i=1}^t e\ell_\lambda(\sigma^i).
 $
\end{enumerate}
    
\end{thm}

\begin{proof}  Item (1) is simply a restatement of Theorem~\ref{thm:poset-topology-mu}, and Item (2) then follows from Part (3) of Theorem~\ref{thm:Segre-t_EL}; for each $t=1,\ldots, n$, $\sigma^i=\lambda(c_i)$ where $(c_1,\ldots,c_t)$ is the $t$-tuple of chains in $P$ corresponding to a decreasing chain in $P^{(t)}$. 
By Part (1) of Theorem~\ref{thm:Segre-t_EL}, 
 the statements follow. 
\end{proof}

Example~\ref{ex:t-Segre-chain-label}  
 gives an EL-labeling for $B_4^{(3)}$ where a maximal chain $C$ is in bijection with a $3$-tuple of maximal chains $c_i$ in $B_4$. The labels for the chains $c_i$ are permutations of $[4]$. In general, for an EL-labeling of a poset $P$, the edge labels in a chain need not be distinct.  The next example illustrates that it is possible for edges in the same chain to have repeated labels. 

\begin{ex}\label{ex:repeat-labels} Figure \textsc{\ref{fig:P-repeat-labels}} is an EL-labeling $\lambda: \mathcal{E}(P)\rightarrow [3]$ of a poset $P$ with repeated edge labels for the right-most maximal chain. The labeling $\lambda^{(2)}$ of the Segre product $P^{(2)}=P\circ P$ shown in Figure \textsc{\ref{fig:PP-repeat-labels}} is also an EL-labeling. The two ordered pairs of words in $[3]\times [3]$ with no common ascents are $(132, 212)$ and $(212, 132)$. Since $e\ell_\lambda(132)=1$ and $e\ell_\lambda(212)=1$, by Theorem \ref{thm:mu-P-to-t-Segre-of-P-rank-sel} Item $(2)$, we have $\mu(P^{(2)})=-2$, which agrees with $\mu(P^{(2)})$ computed directly from the poset. 
The poset $P^{(2)}$ has exactly two decreasing maximal chains $C_1=\hat{0}\coveredby a,b\coveredby d,d \coveredby \hat{1}$ and $C_2=\hat{0}\coveredby b,a\coveredby d,d \coveredby \hat{1}$, shown in Figure \textsc{\ref{fig:PP-repeat-labels}} with thicker edges.

\begin{figure}\
    \begin{subfigure}{0.3\textwidth}
        \centering
        \begin{tikzpicture}
        \node (max) at (0,0) {$\hat{1}$};
        \node (c) at (-1,-1) {$c$};
        \node (d) at (1,-1) {$d$};
        \node (a) at (-1,-2) {$a$};
        \node (b) at (1,-2) {$b$};
        \node (min) at (0,-3) {$\hat{0}$};
        \draw (min)--(a)--(c)--(max)--(d)--(a);
        \draw (min)--(b)--(d);
        \node (c1) at (-0.65,-0.4) {\textcolor{red}{\scriptsize{$\mathbf 3$}}};
        \node (d1) at (0.65,-0.4) {\textcolor{red}{\scriptsize{$\mathbf 2$}}};
        \node (ac) at (-1.15,-1.5) {\textcolor{red}{\scriptsize{$\bf2$}}};
        \node (bd) at (1.15,-1.5) {\textcolor{red}{\scriptsize{$\bf1$}}};
        \node (ad) at (0,-1.3) {\textcolor{red}{\scriptsize{$\bf3$}}};
        \node (0a) at (-0.65,-2.6) {\textcolor{red}{\scriptsize{$\bf1$}}};
        \node (0b) at (0.65,-2.6) {\textcolor{red}{\scriptsize{$\bf2$}}};
    \end{tikzpicture}
    \caption{An EL-labeling of $P$}
    \label{fig:P-repeat-labels}
    \end{subfigure}
    \begin{subfigure}{0.5\textwidth}
        \centering
        \begin{tikzpicture}
        \node (max) at (0,0) {$\hat{1}, \hat{1}$};
        \node (cc) at (-3,-1) {$c,c$};
        \node (cd) at (-1,-1) {$c,d$};
        \node (dc) at (1,-1) {$d,c$};
        \node (dd) at (3,-1) {$d,d$};
        \node (aa) at (-3,-3) {$a,a$};
        \node (ab) at (-1,-3) {$a,b$};
        \node (ba) at (1,-3) {$b,a$};
        \node (bb) at (3,-3) {$b,b$};
        \node (min) at (0,-4) {$\hat{0}, \hat{0}$};
        \draw (min)--(aa)--(cc)--(max)--(cd)--(aa)--(dc)--(max);
        \draw (aa)--(dd);
        \draw (ab)--(cd);
        \draw (ba)--(dc);
        \draw [very thick] (min)--(ab)--(dd)--(max);
        \draw [very thick] (min)--(ba)--(dd);
        \draw (min)--(bb)--(dd);
        \node (cc1) at (-1.6,-0.3) {\textcolor{red}{\rotatebox[origin=c]{18}{\scriptsize{$\mathbf{(3,3)}$}}}};
        \node (cd1) at (-0.35,-0.6) {\textcolor{red}{\rotatebox[origin=c]{46}{\scriptsize{$\mathbf{(3,2)}$}}}};
        \node (dc1) at (0.35,-0.6) {\textcolor{red}{\rotatebox[origin=c]{-46}{\scriptsize{$\mathbf{(2,3)}$}}}};
        \node (dd1) at (1.6,-0.3) {\textcolor{red}{\rotatebox[origin=c]{-18}{\scriptsize{$\mathbf{(2,2)}$}}}};
        \node (aacc) at (-3.4,-2) {\textcolor{red}{\scriptsize{$\mathbf{(2,2)}$}}};
        \node (aacd) at (-2.2,-1.9) {\textcolor{red}{\rotatebox[origin=c]{45}{\scriptsize{$\mathbf{(2,3)}$}}}};
        \node (aacd) at (2.1,-2.2) {\textcolor{red}{\rotatebox[origin=c]{45}{\scriptsize{$\mathbf{(1,3)}$}}}};
        \node (aadc) at (-1.45,-2) {\textcolor{red}{\rotatebox[origin=c]{26}{\scriptsize{$\mathbf{(3,2)}$}}}};
        \node (aadd) at (0.01,-1.85) {\textcolor{red}{\rotatebox[origin=c]{18}{\scriptsize{$\mathbf{(3,3)}$}}}};
        \node (abdd) at (0.45,-2.1) {\textcolor{red}{\rotatebox[origin=c]{26}{\scriptsize{$\mathbf{(3,1)}$}}}};
        \node (abcd) at (-0.63,-1.4) {\textcolor{red}{\scriptsize{$\mathbf{(2,1)}$}}};
        \node (badc) at (1.35,-2.15) {\textcolor{red}{\scriptsize{$\mathbf{(1,2)}$}}};
        \node (bbdd) at (3.4,-2) {\textcolor{red}{\scriptsize{$\mathbf{(1,1)}$}}};
        \node (0aa) at (-1.6,-3.7) {\textcolor{red}{\rotatebox[origin=c]{-18}{\scriptsize{$\mathbf{(1,1)}$}}}};
        \node (0bb) at (1.6,-3.7) {\textcolor{red}{\rotatebox[origin=c]{18}{\scriptsize{$\mathbf{(2,2)}$}}}};
        \node (0ab) at (-0.35,-3.4) {\textcolor{red}{\rotatebox[origin=c]{-46}{\scriptsize{$\mathbf{(1,2)}$}}}};
        \node (0ba) at (0.35,-3.4) {\textcolor{red}{\rotatebox[origin=c]{46}{\scriptsize{$\mathbf{(2,1)}$}}}};
    \end{tikzpicture}
    \caption{EL-labeling of $P\circ P$}
    \label{fig:PP-repeat-labels}
    \end{subfigure}
\caption{An EL-labeling with edges having repeated labels}\label{fig:repeat-labels}
\end{figure}
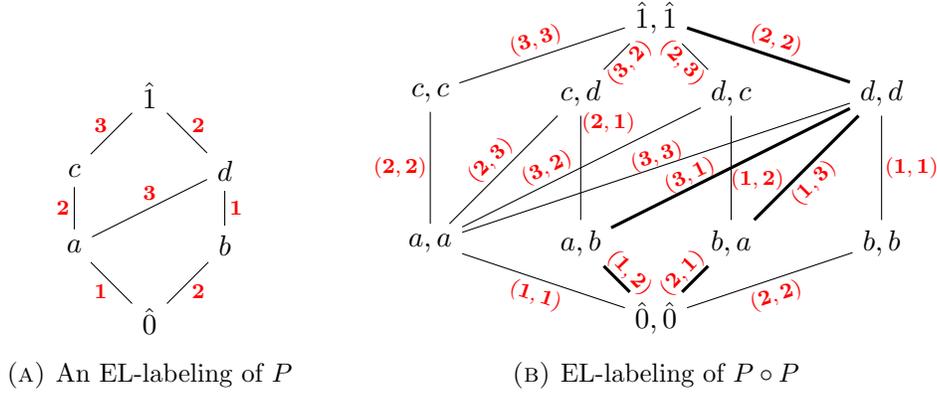
\end{ex}

These formulas recover  some of Stanley's enumerative results for the Boolean lattice as a special case. The combinatorial objects of interest arising from Stanley's theorem are  the 
 $t$-tuples  of permutations in the symmetric group $\sym_n$ with no common ascent, $i$ being an ascent of a permutation $\sigma$ if $\sigma(i)<\sigma(i+1)$. 
For $t=2$ these numbers first appear in work of Carlitz, Scoville and Vaughn \cite{CSV}. 
For arbitrary $t$ they  have also already appeared in the literature; see Abramson and Promislow  \cite{AbramsonPromislowJCTA1978} and $\mathtt{oeis}$ A212855.  

Recall the EL-labeling  for $B_n$ defined by $\lambda(A\coveredby B)= a$ where $B=A\cup\{a\}$. Using this labeling, 
the maximal chains are  in bijection with permutations of $\sym_n$. 
Then Theorem~\ref{thm:mu-P-to-t-Segre-of-P-rank-sel} applies with $\Lambda=[n]$; here $w_n(\Lambda)$ is the set $\mathfrak{S}_n$ of all permutations on $n$ letters, and $e\ell_\lambda(\sigma)=1$ for all $\sigma\in \mathfrak{S}_n$. 
 Theorem~\ref{thm:mu-P-to-t-Segre-of-P-rank-sel}  therefore specializes to the following result of Stanley, stated  with more details in \cite[Proposition 2.5]{LiSu-ECA2025}. 

 For $n\ge 1$, let $w^{(t)}_n$ denote the number of $t$-tuples of permutations in $\sym_n$ with no common ascent. More generally,  for $J\subseteq [n-1]$, let $w^{(t)}_n(J)$ be the number of $t$-tuples of permutations in $\sym_n$ such that their set of common ascents coincides with the complement of $J$ in $[n-1]$.

 \begin{prop}[{See \cite[Theorem 3.1]{RPS-BinomialPosetsJCTA1976}}] \label{prop:mu-Bnt} The M\"obius number of $B^{(t)}_n$  is given by $(-1)^{n} w^{(t)}_n$.  
More generally, for the rank selection $J\subseteq [n-1]$, the M\"obius number  $\mu(B^{(t)}_n(J) )$ of $B^{(t)}_n(J)$ is given by $(-1)^{|J|-1} w^{(t)}_n(J)$.
\end{prop}

\section{EL-labeling for the $t$-fold Segre power of the subspace lattice $B_n(q)$}\label{sec:tfold-subsp-lattice}

In this section we examine the consequences of Theorem~\ref{thm:Segre-t_EL} for the 
$t$-fold Segre power of the subspace lattice $B_n(q)$.

We begin by describing 
the EL-labeling (see Definition \ref{def:EL}) on the subspace lattice $B_n(q)$ given in \cite{YLiqCSV2023}, which, by Theorem~\ref{thm:Segre-t_EL},  extends to an EL-labeling on the $t$-fold Segre power $B^{(t)}_n(q)$.

Let $A$ be the set of all atoms of $B_n(q)$. For a subspace $X\in B_n(q)$ of $\mathbb{F}_q^n$, let $A(X):=\{V\in A|V\leq X\}$. The following two steps define an edge-labeling on $B_n(q)$, which was shown to be an EL-labeling in \cite[Proposition~2.2]{YLiqCSV2023}.
\begin{enumerate}
\item Let $V\in A$ be an atom of $B_n(q)$. It is a $1$-dimensional subspace of $\mathbb{F}_q^n$ with a basis element $v=(v_1,\ldots,v_n)$, $v_i\in \mathbb{F}_q$. Define a map 
$f$: $A \longrightarrow [n]$ by $f(V)=i$, where $i$ is the largest index such that $v_i\ne 0$.  Thus $f(V)$ is the index of the right-most non-zero coordinate of $v$.
\item For a $k$-dimensional subspace $X$ of $\mathbb{F}_q^n$, by Gaussian elimination, there exists a basis of $X$ whose elements have distinct right-most non-zero coordinates. This implies that $f(A(X))$, the image of $A(X)$ under $f$, has $k=\dim(X)$ elements. Let $Y$ be an element of $B_n(q)$ that covers $X$, thus $\dim(Y)=\dim(X)+1$. Then the set $f(A(Y))\setminus f(A(X)) \subset [n]$ has exactly one element. This element will be the label of the edge $(X,\,Y)$. 
\end{enumerate}

\begin{ex}
In the case of $B_4(3)$, if $X=\spn\{\langle 1,0,1,0\rangle,\langle 2,1,0,0\rangle\}$, then $A(X)$ also contains $\spn\{\langle 0,1,1,0\rangle\}$ and $\spn\{\langle 2,2,1,0\rangle\}$. But $X$ does not contain any vector whose right-most non-zero coordinate is the first or the last coordinate. So $f(A(X))=\{2,3\}$ and $|f(A(X))|=2$. Let $Y=\spn\{\langle1,0,1,0\rangle,\langle 2,1,0,0\rangle,\langle 1,0,0,1\rangle\}=\spn\{X,\langle 1,0,0,1\rangle\}$.  Then $f(A(Y))=\{2,3,4\}$, and $f(A(Y))\setminus f(A(X))=\{4\}$. The edge $(X,Y)$ will take label $4$.
\end{ex}

\begin{figure}
\centering
\begin{tikzpicture}
  \node (max) at (0,0) {$\mathbb F_2^3$};
  \node (f) at (-4.5,-1.5) {$\spn\begin{Bmatrix}\langle 1,0,0\rangle,\\
      \langle 0,1,0\rangle\end{Bmatrix}$};
  \node (g) at (-1,-1.5) {$\spn\begin{Bmatrix}\langle 0,1,0\rangle,\\
  \langle 0,0,1\rangle\end{Bmatrix}$};
  \node (h) at (2.5,-1.5) {$\spn\begin{Bmatrix}\langle 1,1,0\rangle,\\
  \langle 1,0,1\rangle\end{Bmatrix}$};
  \node (i) at (5,-1.5) {$\cdots\cdots$};

  \node (a) at (-5, -4) {$\spn\{\langle 1,0,0\rangle\}$};
  \node (b) at (-2, -4) {$\spn\{\langle 0,1,0\rangle\}$};
  \node (c) at (1, -4) {$\spn\{\langle 1,1,0\rangle\}$};
  \node (d) at (4, -4) {$\spn\{\langle 0,1,1\rangle\}$};
  \node (e) at (6,-4) {$\cdots\cdots$};
  \node (min) at (0,-5.5) {$\varnothing$};
  \node (j) at (5,-3) {$\cdots\cdots$};

  \node (f1) at (-2,-0.5) {\textcolor{red}{$\bf3$}};
  \node (g1) at (-0.55,-0.5) {\textcolor{red}{$\bf1$}};
  \node (j1) at (1.2,-0.5) {\textcolor{red}{$\bf1$}};
  \node (af) at (-5,-3) {\textcolor{red}{$\bf2$}};
  \node (bf) at (-3.3,-3) {\textcolor{red}{$\bf1$}};
  \node (bg) at (-1.85,-3.2) {\textcolor{red}{$\bf3$}};
  \node (cf) at (-0.8,-3) {\textcolor{red}{$\bf1$}};
  \node (ch) at (1.3,-3.2) {\textcolor{red}{$\bf3$}};
  \node (dg) at (2.4,-3) {\textcolor{red}{$\bf2$}};
  \node (dh) at (3.6,-3) {\textcolor{red}{$\bf2$}};
  \node (0a) at (-2.5,-5) {\textcolor{red}{$\bf1$}};
  \node (0b) at (-0.7,-4.7) {\textcolor{red}{$\bf2$}};
  \node (0c) at (0.3,-4.7) {\textcolor{red}{$\bf2$}};
  \node (0d) at (1.9,-5) {\textcolor{red}{$\bf3$}};

  \draw (min) -- (a) -- (f)--(max);
  \draw (min)--(b)--(f)--(c)--(min);
  \draw (b)--(g);
  \draw (c)--(h);
  \draw (min)--(d)--(g)--(max);
  \draw (d)--(h)--(max);
  
\end{tikzpicture}
\caption{An EL-labeling of some maximal chains in $B_3(2)$}
\label{fig:EL-lab-B_3(2)}
\end{figure}
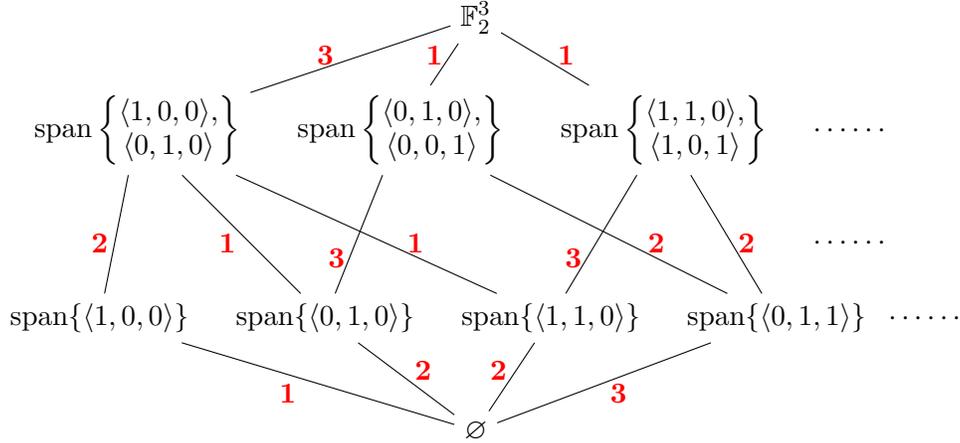
\begin{ex}
 An example of the EL-labeling for $B_3(2)$ defined in this section is shown, partially, in Figure \ref{fig:EL-lab-B_3(2)}. The poset $B_3(2)$ is too large and we can only demonstrate a portion of it in the figure. The left-most chain 
 $$\varnothing\coveredby \spn\{\langle1,0,0\rangle\}\coveredby \spn\{\langle1,0,0\rangle, \langle 0,1,0\rangle\}\coveredby\mathbb F_2^3$$
 is the unique increasing maximal chain with chain label $123$. The two right-most chains are two of the decreasing maximal chains with label $321$. 
\end{ex}

We recall from Section~\ref{sec:SegreEL} the extension of  this EL-labeling on $B_n(q)$  to an EL-labeling on the $t$-fold Segre power $B^{(t)}_n(q)$. See Figure~\ref{fig: EL-lab-Segre-Bnq} (reproduced from Figure 1 in \cite{YLiqCSV2023}) for an example of this labeling for the Segre product $B^{(t)}_n(q)$ when $n=2, t=2$, and $q=2$. Given two elements $X=(X_1,X_2,\dots,X_t)$ and $Y=(Y_1,Y_2,\dots,Y_t)$ in $B^{(t)}_n(q)$ satisfying the covering relation $X\coveredby Y$, we must have $X_i\coveredby Y_i$ in $B_n(q)$ for all $i\in[t]$. Let $\lambda_i$ be the edge label for $(X_i,Y_i)$ in the EL-labeling of $B_n(q)$. The edge $(X,Y)$ in $B^{(t)}_n(q)$ will take the label $(\lambda_1, \lambda_2, \dots, \lambda_t)$ in 
 the product poset $\underbrace{[n]\times\cdots\times [n]}_t$, that is,  with component-wise order relation $<_t$  defined by  $(a_1,\ldots,a_t)<_t (b_1,\ldots,b_t)$ if $a_i\le b_i$ for all $i=1, \ldots , t$, with  $(a_1,\ldots,a_t)=(b_1,\ldots,b_t)$ if $a_i= b_i$ for all $i=1, \ldots , t$.

\begin{figure}
\centering
\begin{tikzpicture}
  \node (max) at (0,4) {$(\mathbb F_2^2,\mathbb F_2^2)$};
  \node (a) at (-4,2) {$\begin{pmatrix}
		\spn\{\langle 1,0\rangle\},\\
	\spn\{\langle 1,0\rangle\}											\end{pmatrix}$};
  \node (b) at (-1,2) {$\begin{pmatrix}
	\spn\{\langle 1,0\rangle\},\\
	\spn\{\langle 0,1\rangle\}											\end{pmatrix}$};
  \node (c) at (1.5,2) {\dots\dots};  
  \node (e) at (4,2) {$\begin{pmatrix}
	\spn\{\langle 1,1\rangle\},\\
	\spn\{\langle 1,1\rangle\}										
		\end{pmatrix}$};
  \node (f) at (-2.1,0.7) {\textcolor{red}{$\mathbf{(1,1)}$}};
  \node (g) at (-0.07,1) {\textcolor{red}{$\mathbf{(1,2)}$}};
  \node (i) at (2.1,0.7) {\textcolor{red}{$\mathbf{(2,2)}$}};
	\node (j) at (-2.1,3.3) {\textcolor{red}{$\mathbf{(2,2)}$}};
  \node (k) at (-0.07,3) {\textcolor{red}{$\mathbf{(2,1)}$}};
	\node (m) at (2.1,3.3) {\textcolor{red}{$\mathbf{(1,1)}$}};
  \node (min) at (0,0) {$(\varnothing,\varnothing)$};
  \draw (min) -- (a) -- (max) -- (b) -- (min)-- (e) -- (max);
\end{tikzpicture}
\caption{An EL-labeling of $B_2(2)\circ B_2(2)$}
\label{B_2(2)}
\label{fig: EL-lab-Segre-Bnq}    
\end{figure}
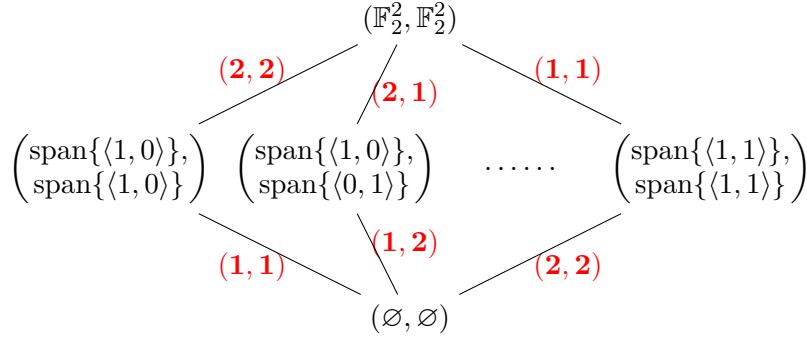

\begin{prop} \label{prop:t_EL} 
The edge-labeling of $B^{(t)}_n(q)$ defined above is an EL-labeling.
\end{prop}

\begin{proof}  This follows from Theorem~\ref{thm:Segre-t_EL}.
\end{proof}

 With this EL-labeling, as in the case of $B_n$, Theorem~\ref{thm:mu-P-to-t-Segre-of-P-rank-sel} specializes  to  Stanley's result on the subspace lattice {\cite[Theorem~3.1]{RPS-BinomialPosetsJCTA1976}}, also restated in \cite[Theorem 2.4]{LiSu-ECA2025}.
 Proposition~\ref{prop:mu-Bnt} is the special case $q=1$.  

\begin{thm}[{\cite[Theorem~3.1]{RPS-BinomialPosetsJCTA1976}}] \label{thm:rank-sel-beta-inv-N-of-q} Let $J\subseteq [n-1]$.  Write $J^c=[n-1]\setminus J$, and $\Asc(\sigma)$ for the set of ascents of a permutation $\sigma$. Then 
\[\tilde\beta_{B_n^{(t)}(q)}(J)
=\sum_{\substack
{(\sigma^1,\ldots,\sigma^t)\in\mathfrak{S}_n^{\times t}\\
{J^c=\cap_{i=1}^t \Asc(\sigma^i)}
}} \prod_{i=1}^t q^{\inv(\sigma^i)}\]
The M\"obius number of the $t$-fold Segre power of the subspace lattice is $(-1)^{n-2}W^{(t)}_n(q)$, where 
\[ W^{(t)}_n(q)
:=\sum_{(\sigma_1,\ldots \sigma_t)\in  \sym_n^{\times^t}}
\prod_{i=1}^t q^{\inv(\sigma_i)},\] 
and the sum is over all $t$-tuples of permutations in $\sym_n$ with no common ascent.
\end{thm}

The case $t=2$ was rediscovered in \cite[Lemma 2.3]{YLiqCSV2023} by counting decreasing chains in the EL-labeling described above. 

\section{Conclusion}

Theorem~\ref{thm:mu-P-to-t-Segre-of-P-rank-sel} gives a formula for the rank-selected invariants of $P^{(t)}$ for any EL-shellable  poset $P$. The lattice $\Pi_n$ of set partitions of $[n]$ is an example of a poset with many interesting EL-labelings (see \cite[Section 3.2.2]{WachsPosetTop2007}). In a future paper we plan to investigate rank-selected invariants for the Segre powers of $\Pi_n$. 

The subspace lattice is a $q$-analogue of the Boolean lattice in the sense of Simion \cite{Simion1995}.  We propose to continue this project using Simion's framework, and  investigate the implications  of the present  work for other pairs of posets $(P, Q)$, where $Q$ is a $q$-analogue of $P$.  

\section{Funding information}
 This material is based upon work supported by the National Science Foundation under Grant No. DMS-1928930, while the authors were in residence at the Simons Laufer Mathematical Sciences Research Institute in Berkeley, California, during the  Summer of 2023.


\end{document}